\documentclass[notitlepage, 11pt]{article}
\usepackage[shortlabels]{enumitem}
\usepackage{amsmath}
\usepackage{amssymb}
\usepackage{amsthm}
\usepackage{abstract}
\usepackage{xcolor}
\usepackage{comment}
\usepackage{mathtools}
\usepackage{graphicx}
\usepackage{caption}
\usepackage{subcaption}
\usepackage{float}
\usepackage{enumitem}
\setlist[itemize]{noitemsep}
\setlist[enumerate]{noitemsep}
\usepackage{hyperref}
\usepackage{cleveref}
\hypersetup{
    colorlinks=true,
    linkcolor=blue,
    filecolor=magenta,      
    urlcolor=cyan
    }
\makeatletter
\newcommand*{\barfix}[2][.175ex]{%
  \mathpalette{\@barfix{#1}}{#2}%
}
\newcommand*{\@barfix}[3]{%
  \vbox{%
    \kern#1\relax
    \hbox{$#2#3\m@th$}%
  }%
}
\makeatother

\newtheorem{theorem}{Theorem}
\newtheorem{thm}{Theorem}[section]

\newtheorem{lemma}[thm]{Lemma}

\theoremstyle{definition}

\newtheorem{remark}[thm]{Remark}

\newcommand{\footremember}[2]{%
    \footnote{#2}
    \newcounter{#1}
    \setcounter{#1}{\value{footnote}}%
}
\newcommand{\footrecall}[1]{%
    \footnotemark[\value{#1}]%
} 

\usepackage[margin=1in]{geometry}

\title{\vspace{-1.5cm}A large hole in pseudo-random graphs}
\author{%
\and Sahar Diskin \footremember{alley1}{\scriptsize{School of Mathematical Sciences, Tel Aviv University, Tel Aviv 6997801, Israel. Emails: sahardiskin@mail.tau.ac.il, krivelev@tauex.tau.ac.il, markbreit@mail.tau.ac.il.}}%
\and Michael Krivelevich \footrecall{alley1}%
\and Itay Markbreit \footrecall{alley1}
\and Maksim Zhukovskii \footremember{alley2}{\scriptsize{School of Computer Science, University of Sheffield, UK. Email: m.zhukovskii@sheffield.ac.uk.}}%
}
\begin{document}
\date{\vspace{-1cm}}

\maketitle
\begin{abstract}
We show that there exist constants $\delta_1,\delta_2>0$ such that if $G$ is an $(n,d,\lambda)$-graph with $\lambda/d\le\delta_1$, then $G$ contains an induced cycle of length at least $\delta_2n/d$. We further demonstrate that, up to a constant factor, this is best possible. Utilising our techniques, we derive that the number of non-isomorphic induced subgraphs of such $G$ is at least exponential in $n\log d/d$, and further demonstrate that this is tight up to a constant factor in the exponent.
\end{abstract}

\section{Introduction and main results}
The study of extremal problems about induced subgraphs is a popular theme in combinatorics. Of particular interest is finding the lengths of a longest induced path and of a longest induced cycle (also called a \textit{hole}) in a given graph $G$. Two key examples, that have been studied in the past, are when $G$ is the $d$-dimensional hypercube, wherein finding the length of a longest induced path is known as the `Snake-in-the-Box' problem~\cite{AK88, Z97}, and when $G$ is the binomial random graph $G(n,p)$. Let us note here that the length of a longest induced path is at most twice the size of a largest independent set in $G$.

The study of induced cycles in $G(n,p)$ dates back to the late '80s. Frieze and Jackson~\cite{FJ87} showed that for each sufficiently large constant $d$, \textbf{whp}\footnote{With high probability, that is, with probability tending to one as $n$ tends to infinity} the random graph $G(n,d/n)$ contains an induced cycle of length at least $c(d)\,n$ for some constant $c(d)>0$. Łuczak~\cite{L93} and, independently, Suen~\cite{S92} later improved this, proving that whenever $d>1$, \textbf{whp} $G(n,d/n)$ contains an induced cycle of length at least $(1+o(1))\frac{\ln d}{d}\,n$. A simple first‑moment argument implies that \textbf{whp} the length of a longest induced cycle for large $d$ in $G(n,d)$ is at most $(1+o_d(1))\frac{2\ln d}{d}\,n$. This upper bound was shown to be asymptotically tight by Draganić, Glock, and Krivelevich~\cite{DGK22}. Dutta and Subramanian~\cite{DS23} further established a two-point concentration result on the length of a longest induced path in $G(n,p)$ for $p\ge \log^2n/\sqrt{n}$.

It appears natural to extend the study of extremal problems for induced paths and cycles for random to \textit{pseudo-random} graphs. The latter can be informally described as graphs whose edge distribution resembles that of a truly random graph $G(n,p)$ of a similar density. Formally, given a $d$-regular graph $G$ on $n$ vertices, denote the eigenvalues of its adjacency matrix by $d=\lambda_1\ge \lambda_2\ge \cdots \ge \lambda_n$. Letting $\lambda\coloneqq \max\{|\lambda_2|,|\lambda_n|\}$, we then say that $G$ is an $(n,d,\lambda)$-graph. The expander mixing lemma, due to Alon and Chung~\cite{AC88}, relates the \textit{spectral ratio} $\frac{\lambda}{d}$ to the edge-distribution of these graphs (see Lemma~\ref{l: eml}). We refer the reader to~\cite{HLW06,K19,KS06} for comprehensive surveys on the subject of pseudo-random graphs and expanders.  

As noted before, it is thus quite natural to ask whether one can find long induced cycles, similar to the case of the binomial random graph $G(n,d/n)$, in $(n,d,\lambda)$-graphs (naturally under some assumption on the aforementioned spectral ratio). A first result in this direction was recently obtained by Dragani\'c and Keevash~\cite{DK25}, who showed the following: any $(n,d,\lambda)$-graph $G$ with $\lambda<d^{3/4}/100$ and $d<n/10$ contains an induced path of length $\frac{n}{64d}$; their paper did not address the problem of long induced cycles.

Our first main result improves upon the result of Dragani\'c and Keevash~\cite{DK25} by significantly relaxing the assumption on the spectral ratio, as well as showing the existence of a long induced \textit{cycle}, instead of a path.
\begin{theorem}\label{th: large hole}
There exist constants $\delta_1,\delta_2,\delta_3>0$ such that the following holds. For any integers $n, d$ and for any $(n,d,\lambda)$-graph $G$ such that $d\le \delta_3 n$ and $\lambda/d\le \delta_1$, $G$ contains an induced cycle of length at least $\delta_2n/d$.
\end{theorem}
A few comments are in place. We note that for $\delta_1$ sufficiently small, one can take $\delta_2=1/150$; in fact, as we will soon see, we can find a longer induced \textit{path}, of length at least $\frac{n}{48d}$, thus improving also the constant factor in the result of~\cite{DK25}. Further, since a random $d$-regular graph on $n$ vertices $G_{n,d}$ is typically an $(n,d,\lambda)$-graph with $\lambda\le 2\sqrt{d}$ (see, e.g.,~\cite{F08}), we obtain that for large enough $d$, \textbf{whp} $G_{n,d}$ contains an induced cycle of length $\Omega(n/d)$. This improves the bound of Frieze and Jackson~\cite{FJ87} from 1985, who showed that \textbf{whp} $G_{n,d}$ contains an induced cycle of length $\Omega(n/d^2)$ (see also~\cite{EFMN21}), and makes progress towards resolving~\cite[Problem 74]{F19}, which asks to determine the typical length of a longest induced cycle in $G_{n,d}$. Finally, let us note that our proof yields a randomised algorithm, which finds in a linear in $n$ time an induced cycle of length at least $\delta_2n/d$ \textbf{whp}.

It turns out, perhaps somewhat surprisingly, that insights and results from the setting of \textit{site percolation} prove very efficient, leading to a rather simple (in hindsight) proof of Theorem \ref{th: large hole}. Intuitively, finding an induced cycle is 'easier' when the graph is sparse, and in this case, choosing a random set of vertices typically yields a sparser graph, wherein every induced structure is also an induced structure in the host graph. Formally, given a host graph $G=(V,E)$, form a random subset $V_p$ by retaining every $v\in V$ independently with probability $p$. The $p$-site-percolated\footnote{In many papers, this notation is reserved for the \textit{bond}-percolated random subgraph. Throughout this paper, we reserve $G_p$ for the $p$-\textit{site}-percolated random subgraph.} subgraph $G_p$ is then $G_p\coloneqq G[V_p]$. We will derive Theorem \ref{th: large hole} from our next result.
\begin{theorem}\label{th: percolated path}
Let $\epsilon>0$ be a sufficiently small constant. Let $n,d\coloneqq d(n)\in\mathbb{N}$ be such that $d=o(n)$. Let $p=\frac{1+\epsilon}{d}$. Then, there exists a constant $\delta\coloneqq \delta(\epsilon)>0$ such that the following holds. Let $G$ be an $(n,d,\lambda)$-graph with $\frac{\lambda}{d}\le\delta$. Then, \textbf{whp} $G_p$ contains an induced path of length at least $\frac{\epsilon^2n}{3d}$.
\end{theorem}
Note that Theorem \ref{th: percolated path} naturally implies that \textit{deterministically} the whole graph $G$ contains a long induced path. In fact, with a little more effort, one can typically find an induced cycle of length $\Omega(n/d)$ in $G_p$, see details in Remark \ref{remark: cycle in Gp}. Further, as we mentioned after Theorem \ref{th: large hole}, it is not hard to verify that one can take any $\epsilon\le1/4$, and thus obtain an induced path of length $\frac{n}{48d}$.

Since the edge-distribution of an $(n,d,\lambda)$-graph $G$ (when $\lambda/d\le \delta$ for some sufficiently small constant $\delta$) resembles that of the binomial random graph $G(n,d/n)$, one might expecto to find an induced path (or even cycle) of length $\Theta(\frac{\ln d}{d})n$ in $G$. Perhaps somewhat surprisingly, as our next result shows, this is not the case.
\begin{theorem}\label{th: construction}
For every constant $\delta>0$, there exists a constant $C\coloneqq C(\delta)>0$ such that the following holds. For every sufficiently large $d\in \mathbb{N}$, there exist infinitely many $n$ for which there exists an $(n,d,\lambda)$-graph $G$ with $\frac{\lambda}{d}\leq\delta$, whose longest induced path is of length at most $\frac{Cn}{d}$.
\end{theorem}
In particular, Theorem~\ref{th: construction} shows that Theorem~\ref{th: large hole} is tight up to a multiplicative constant. Let us note that this also marks a key difference between a longest path in an $(n,d,\lambda)$-graph, which is of length linear in $n$~\cite{KS13}, and a longest induced path, which we now see might be of length at most linear in $n/d$. 

Given an $(n,d,\lambda)$-graph $G$ with $\lambda/d$ being a (small) constant, by Theorems \ref{th: large hole} and \ref{th: construction} we know that a largest hole of $G$ is of size $\Omega(n/d)$, and this estimate is tight. It would be interesting to relate the spectral ratio $\lambda/d$ (where we allow the ratio to depend on $d$) to the size of a largest hole of an $(n,d,\lambda)$-graph $G$. In particular, one may wonder whether a largest hole is of size $\Omega(n\log (d/\lambda)/d)$ --- indeed, when $\lambda\ll d$, the independence number of an $(n,d,\lambda)$-graph $G$ satisfies $\alpha(G)=\Omega(n\log(d/\lambda)/d)$ (see, for example, \cite[Proposition 4.6]{KS06}).

\paragraph{Application}
Theorem \ref{th: percolated path} and its proof bear interesting consequences in terms of counting non-isomorphic induced subgraphs. Let $\mu(G)$ be the number of non-isomorphic induced subgraphs in $G$. Erd\H{o}s and R\'enyi conjectured that for every constant $c_1>0$, there exists a constant $c_2>0$ such that if $G$ has no subset $S$ of $c_1\log n$ vertices on which $G[S]$ is either the complete graph or the empty graph (that is, $G$ is $c_1$-Ramsey), then $\mu(G)\ge \exp\{c_2n\}$. Several results in this direction were obtained~\cite{AB89,AH91}; in particular, in 1976  M\"uller~\cite{M79} showed that \textbf{whp} $\mu(G(n,1/2))=2^{(1-o(1))n}$. The conjecture was finally confirmed by Shelah in 1998~\cite{S92}. 

In a recent work~\cite{KZ25}, the second and fourth authors determined the asymptotics of $\mu(G(n,p))$ for (almost) the entire range of $G(n,p)$. They further showed that \textbf{whp} the number of non-isomorphic induced subgraphs in a random $d$-regular graph $G_{n,d}$ is exponential in $n$ as well, and the base of the exponent grows to $2$ with growing $d$. Utilising Theorem \ref{th: percolated path}, we are able to show that for pseudo-random graphs, $\mu(G)$ is exponential in $(n\log d)/d$.
\begin{theorem}\label{th: isomorphism lower bound}
There exist constants $\delta_1,\delta_2,\delta_3>0$ such that the following holds. For any integers $n, d$ and for any $(n,d,\lambda)$-graph $G$ such that $d\le \delta_3 n$ and $\lambda/d\le \delta_1$, $\mu(G)\ge \exp\left(\frac{\delta_2n\log d}{d}\right)$.
\end{theorem}

We further show that the above is tight up to the constant in the exponent.
\begin{theorem}\label{th: isomorphism upper bound}
For every constant $\delta>0$, there exists a constant $C>0$ such that the following holds. For every sufficiently large $d\in \mathbb{N}$, there exist infinitely many $n$ for which there exists an $(n,d,\lambda)$-graph $G$ with $\lambda/d\le \delta$ satisfying that $\mu(G)\le \exp\left\{\frac{Cn\log d}{d}\right\}$.
\end{theorem}

\paragraph{Structure of the paper}
In Section \ref{sec: prelim} we present notation used throughout the paper, describe an adaptation of the Depth-First-Search algorithm which we will employ, and collect several lemmas to be utilised in the proof. Then, in Section \ref{sec: proof} we prove Theorem \ref{th: percolated path} and derive Theorem \ref{th: large hole} from it. In Section \ref{sec: isomorphism} we prove Theorem \ref{th: isomorphism lower bound}. Finally, in Section \ref{sec: construction} we give constructions proving Theorems \ref{th: construction} and \ref{th: isomorphism upper bound}.

\section{Preliminaries}\label{sec: prelim}
Given a graph $G=(V,E)$ and sets $A,B\subseteq V(G)$, we denote by $e(A,B)$ the number of edges with one endpoint in $A$ and the other endpoint in $B$. We further abbreviate $e(A)\coloneqq e(A,A)/2=e(G[A])$. We denote by $N(A)$ the external neighbourhood of $A$, that is, $$N(A)=\{v\in V\setminus A\colon \exists u\in A, uv\in E\}.$$ 
Given $v\in V$ and $A\subseteq V$, we denote by $d(v,A)$ the number of neighbours of $v$ in $A$. When $A=V$, we write $d(v)\coloneqq d(v,V)$ for the degree of $v$ in $G$.
Recall that given $G=(V,E)$, we form $V_p$ by retaining every $v\in V$ independently and with probability $p$; we then abbreviate $G_p=G[V_p]$. Throughout the paper, we systematically ignore rounding signs as long as it does not affect the validity of our arguments.

We will make use of the following fairly standard Chernoff-type probability bound (see, for example, Appendix A in~\cite{zbMATH06566409}).
\begin{lemma}\label{lemma:binomial-bounds}
Let $X\sim Bin(n,p)$. Then, for any $0\le t \le np$,
    \[
        \mathbb{P}\left(\left|X- \mathbb{E}[X]\right|\ge t\right)\le 2\exp\left\{-\frac{t^2}{3np}\right\}.
    \]
\end{lemma}

Throughout the rest of the section, we set $\epsilon>0$ to be a sufficiently small constant. We assume that $G$ is an $(n,d,\lambda)$-graph (in fact, we consider a sequence of pairs $(d_k,n_k)\in \mathbb{N}^2$ and a sequence of $(n_k,d_k,\lambda_k)$-graphs $(G_k)_{k\in\mathbb{N}}$ satisfying $d_k=o(n_k)$). We write $\delta=\frac{\lambda}{d}$ for the spectral ratio. We further set $p=\frac{1+\epsilon}{d}$, and let $V_p$ and $G_p$ be as defined above.

Let us first collect several lemmas that will be useful for us throughout the paper. The first is the aforementioned expander mixing lemma, due to Alon and Chung~\cite{AC88}.
\begin{lemma}\label{l: eml}
    For any pair of subsets $A,B\subseteq V(G)$,
    $$\left|e(A,B)-\frac{d}{n}|A||B|\right|\leq \lambda\sqrt{|A||B|}.$$
\end{lemma}

We also require certain results on site percolation on pseudo-random graphs. The first result relates the spectral ratio of $G$ to the vertex-expansion (in $G$) of sets which lie in $G_p$.
\begin{lemma}[Lemma 2.4 in~\cite{zbMATH07751060}, see also~\cite{zbMATH06534995}] \label{expanding}
Let $\alpha\coloneqq\alpha(\epsilon)\in(0,\epsilon^{8})$ be a constant. Suppose that $\delta\le \alpha^{2/\alpha}$. Then, \textbf{whp} $V_p$ does not contain a set $S$ with $|S| = m$, $\frac{\alpha n}{d}\leq m\leq \frac{n}{3d}$, such that $|N_G(S)| < (1-\alpha)\left(dm-\frac{d^2m^2}{2n}\right)$.
\end{lemma}

Let us now recall the notion of \textit{excess}. For a connected graph $H$, we define the excess of $H$ as $\mathrm{exc}(H)\coloneqq |E(H)|-|V(H)| + 1$. If $H$ has more than one connected component, then we set $\mathrm{exc}(H)$ to be the sum of the excesses of each of its components. It will be of use for us to estimate the excess of $G_p$. To that end, we require some additional results on site percolation on pseudo-random graphs.

Given $G_p$, we denote by $L_1$ the largest component of $G_p$. Let us further define $x$ to be the unique solution in $(0,1)$ of
\begin{equation}\label{eq: def of x}
    x=(1+\epsilon)(1-\exp\left\{-x\right\}).
\end{equation}
We note that $x=2\epsilon-\frac{2\epsilon^2}{3}+O(\epsilon^3)$ (see~\cite[Equation (4)]{zbMATH07751060}). The following theorem estimates the typical order of $L_1$.
\begin{thm}[Theorem 2 of~\cite{zbMATH07751060}]\label{DK1}
     Let $\alpha\coloneqq\alpha(\epsilon)\in(0,\epsilon^{8})$ be a constant. Suppose that $\delta\le \alpha^{2/\alpha}$. Then, \textbf{whp},
     $$\left||V(L_1)|-\frac{xn}{d}\right|\leq \frac{7\alpha n}{d},$$
     where $x$ is as in \eqref{eq: def of x}.
\end{thm}
The next result estimates the typical number of edges in $L_1$.
\begin{thm}[Theorem 4 of~\cite{zbMATH07751060}] \label{DK2}
    Let $\alpha\coloneqq\alpha(\epsilon)\in(0,\epsilon^{8})$ be a constant. Suppose that $\delta\le \alpha^{2/\alpha}$. Then, \textbf{whp},
     $$\left|e(L_1)-\frac{((1+\epsilon)^2-(1+\epsilon-x)^2)n}{2d}\right|\leq \frac{8\alpha^{1/4} n}{d},$$
     where $x$ is as in \eqref{eq: def of x}.
\end{thm}
Finally, the following result estimates the typical number of edges in $G_p$ which are neither in $L_1$ nor in isolated trees.
\begin{thm}[Lemma 6.4 of~\cite{zbMATH07751060}] \label{DK3}
     Let $\alpha\coloneqq\alpha(\epsilon)\in(0,\epsilon^{8})$ be a constant. Suppose that $\delta\le \alpha^{2/\alpha}$. Then, \textbf{whp}, the number of edges in $G_p$ which are in components that are neither the giant component nor isolated trees is at most $\frac{7\alpha^{1/4} n}{d}$.
\end{thm}

With these three results at hand, we can now estimate the typical excess of $G_p$. 
\begin{lemma} \label{exc}
Suppose that $\delta\le \epsilon^{24/\epsilon^{12}}$. Then, \textbf{whp}, $\mathrm{exc}(G_p)= O(\epsilon^3)n/d$.
\end{lemma}
\begin{proof}
Note that under the above assumption, we can take $\alpha$ from Theorems \ref{DK1}, \ref{DK2}, and \ref{DK3} to be $\alpha=\epsilon^{12}$. Then, by Theorems \ref{DK1} and \ref{DK2}, 
\begin{align*}
    \mathrm{exc}(G_p[L_1])&\le \frac{((1+\epsilon)^2-(1+\epsilon-x)^2)n}{2d}+\frac{8\epsilon^{12/4} n}{d}-\frac{xn}{d}+\frac{7\epsilon^{12} n}{d}+1\\
    &\le \frac{\left(\epsilon x-x^2/2\right)n}{d}+\frac{O(\epsilon^3)n}{d}=\frac{O(\epsilon^3)n}{d},
\end{align*}
where the last equality follows from $\epsilon x-x^2/2=O(\epsilon^3)$. Since isolated trees have no excess, by the above together with Theorem \ref{DK3}, we conclude that \textbf{whp}
\begin{align*}
    \mathrm{exc}(G_p)\le \mathrm{exc}(G_p[L_1])+\frac{7\epsilon^{12/4}n}{d}=\frac{O(\epsilon^3)n}{d},
\end{align*}
as required.
\end{proof}

We conclude this section with a variant of the Depth First Search (DFS) algorithm, which we will utilise when proving Theorem \ref{th: percolated path}. The variant combines ideas of the algorithm presented in~\cite{DGK22} together with the one presented in~\cite{zbMATH06534995}.

The algorithm is fed a graph $G=(V,E)$ with an ordering $\sigma$ on its vertices, and a sequence $(X_v)_{v\in V}$ of i.i.d. Bernoulli$(p)$ random variables (with $0\le p\le 1$). We maintain five sets of vertices: $T$, the set of vertices yet to be processed; $U\subseteq V_p$, the set of active vertices, kept in a stack (the last vertex to enter $U$ is the first to leave); $S_1$ and $S_2$, the sets of processed vertices (which fell into $V_p$); and $W$, the set of processed vertices which fell outside of $V_p$. We initialise $U,S_1,S_2,W=\varnothing$ and $T=V$. The algorithm terminates once $U\cup T=\varnothing$. As we will see, throughout the execution, $U$ will span an induced path in $G_p$. 

Each step of the algorithm corresponds to exposing a random variable $X_v$, and proceeds as follows.
\begin{enumerate}
    \item If $U$ is empty, we consider the first vertex $v$ in $T$ according to $\sigma$.\label{first}
    \begin{enumerate}
        \item If $X_v=1$, we move $v$ from $T$ to $U$.
        \item Otherwise (that is, if $X_v=0$), we move $v$ from $T$ to $W$.
    \end{enumerate}
    \item If $U\neq \varnothing$, let $u$ be the vertex on the top of the stack $U$.
    \begin{enumerate}
        \item If $u$ has no neighbours in $T$, move $u$ from $U$ into $S_1$, and return to \ref{first} and proceed.
        \item Otherwise, let $v$ be the first vertex in $T$ according to $\sigma$ such that $uv\in E$.
        \begin{enumerate}
            \item If $X_v=0$, we move $v$ from $T$ to $W$.
            \item If $X_v=1$ and $v$ has a neighbour in $U\setminus \{u\}$, we move $v$ from $T$ to $S_2$.
            \item Otherwise (that is, if $X_v=1$ and $v$ has no neighbours in $U\setminus\{u\}$), we move $v$ from $T$ to $U$.
        \end{enumerate}
    \end{enumerate}
\end{enumerate}
We will make use of the following simple observations about the above algorithm. First, at every step of the algorithm, $G_p[U]$ spans an induced path. Further, at every step, there are no edges between $S_1$ and $T$, and thus $N_G(S_1)\subseteq S_2\cup U\cup  W$. Moreover, for every integer $0\le k\le n$, after $k$ steps $|S_1\cup S_2\cup U\cup W|=k$.

Finally, observe that at any step, the connected component $C$ of $G_p$ currently explored (that is, the one containing vertices in $U$) stays connected when restricted to $S_1\cup U$, and every vertex of $C$ in $S_2$ sends at least two edges to $S_1\cup U$. Therefore, $|S_2|\le \mathrm{exc}(G_p)$.

\section{Existence of induced paths and cycles}\label{sec: proof}
We begin by finding an induced path in the percolated graph $G_p$.

\begin{proof}[Proof of Theorem \ref{th: percolated path}]
Let $\delta = \epsilon^{24/\epsilon^{12}}$. Run the DFS algorithm described in Section \ref{sec: prelim} on $G$. We claim that after $\epsilon n$ steps in the algorithm, \textbf{whp}, $|U|\geq \frac{\epsilon^2n}{3d}$, and thus, as $U$ forms an induced path in $G_p$, the statement follows. 
    
Indeed, after $\epsilon n$ steps, $|U\cup S_1\cup S_2|\sim Bin(\epsilon n,p)$. Thus, by Lemma~\ref{lemma:binomial-bounds}, with probability at least $1-2\exp\left\{-\frac{(n/d)^{4/3}}{4\epsilon n /d}\right\}$ we have that $\bigg||U_1\cup S_1\cup S_2|-\frac{(1+\epsilon)\epsilon n}{d}\bigg|\le (n/d)^{2/3}$. Further, by Lemma~\ref{exc} together with our assumption on $\delta$, we have that \textbf{whp} $|S_2|=O(\epsilon^3)n/d$. Assume towards contradiction that $|U|< \frac{\epsilon^2n}{3d}$. Then, \textbf{whp},
\begin{align*}
    |S_1|\ge \frac{(1+\epsilon)\epsilon n}{d}-n^{2/3}-O(\epsilon^3)\frac{n}{d}-\frac{\epsilon^2 n}{3d}\ge \frac{(\epsilon +3\epsilon^2/5)n}{d}.
\end{align*}
On the other hand, \textbf{whp},
$$|S_1|\le |U\cup S_1\cup S_2|\le 4\epsilon n/(3d)<n/(3d),$$ 
and thus by Lemma~\ref{expanding} applied with $\alpha=\epsilon^{12}$ (which is possible by our assumption on $\delta$), we have that \textbf{whp}
\begin{align*}
    |N_G(S_1)|\ge (1-\epsilon^{12})\left(\epsilon+3\epsilon^2/5-\frac{1}{2}(\epsilon+3\epsilon^2/5)^2\right)n>\epsilon n,
\end{align*}
a contradiction, since $N_G(S_1)\subseteq S_2\cup U\cup W$ and $|S_2\cup U\cup W|\le |S_1\cup S_2\cup U\cup W|\le \epsilon n$.
\end{proof}

With this result at hand, we are ready to prove Theorem \ref{th: large hole}.
\begin{proof}[Proof of \Cref{th: large hole}]
Let $\epsilon,\delta(\epsilon)>0$ be as in the statement of Theorem \ref{th: percolated path}. We may assume that $\delta_1\le \delta(\epsilon)$. Let $p=\frac{1+\epsilon}{d}$. Then, by Theorem \ref{th: percolated path}, \textbf{whp} $G_p$ contains an induced path of length at least $\frac{\epsilon^2n}{3d}$. Thus, deterministically, $G$ contains an induced path $P$ on $k=\frac{\epsilon^2n}{3d}$ vertices (note that for the latter to hold, we merely needed the first statement to hold with positive probability).

Let $P=\{v_1,\ldots,v_{k}\}$. Let $P_1=\{v_1,\ldots, v_{k/3}\}$ be the first $\frac{\epsilon^2n}{9d}$ vertices of $P$, and let $P_2=\{v_{2k/3+1},\ldots,v_{k}\}$ be the last $\frac{\epsilon^2n}{9d}$ vertices of $P$. Let $P'=\{v_{k/2-k/20+1},\dots,v_{k/2+k/20}\}$ be the set of $\frac{\epsilon^2n}{30d}$ vertices at the middle of $P$. By \Cref{expanding}, $|N_G(P_1)|, |N_G(P_2)|\geq \frac{\epsilon^2n}{10}$. Since $G$ is $d$-regular, $|N_G(P')|\leq d|P'|=\frac{\epsilon^2n}{30}$. Set $N_1=N_G(P_1)\setminus \left(P\cup N_G(P')\right)$ and similarly $N_2=N_G(P_2)\setminus \left(P\cup N_G(P')\right)$. Note that $|N_1|,|N_2|\geq \frac{\epsilon^2n}{30}$. In particular, by Lemma \ref{l: eml}, $e(N_1,N_2)\ge \epsilon^5 dn>0$.

First, suppose that $N_1\cap N_2\neq \varnothing$. Let $u\in N_1\cap N_2$. Let $w_1$ be the neighbour of $u$ on $P_1$ closest to $P'$, that is, to $v_{k/2-k/20+1}$. Similarly, let $w_2$ be the neighbourhood of $u$ closest to $P'$, that is, to $v_{k/2+k/20}$. Then, $u$ together with the subpath of $P$ starting at $w_1$ and ending at $w_2$ forms an induced cycle in $G$ whose length is at least $|P'|=\frac{\epsilon^2n}{30d}$.

Otherwise, we may assume $N_1\cap N_2=\varnothing$. Let $u_1\in N_1$, $u_2\in N_2$ be such that $u_1u_2\in E(G)$. Let $w_1$ be the neighbour of $u_1$ on $P_1$ closest to $P'$, noting that $u_1$ has no neighbours in $P'\cup P_2$. Similarly, let $w_2$ be the neighbour of $u_2$ closest to $P'$, noting that $u_2$ has no neighbours in $P'\cup P_1$. All that is left is to observe that $u_1u_2$ together with the subpath of $P$ starting at $w_1$ and ending at $w_2$ forms an induced cycle in $G$ whose length is at least $|P'|=\frac{\epsilon^2n}{30d}$. This completes the proof, with $\delta_2=\epsilon^2/30$.
\end{proof}
\begin{remark}\label{remark: cycle in Gp}
We note that, with a bit more effort, one can show the typical existence of an induced cycle of length $\Omega(\epsilon^2n/d)$ in $G_p$. Let us give a sketch of the proof. We may employ a sprinkling argument, setting $p_2=\frac{\epsilon^3}{d}$, and $p_1$ to be such that $(1-p_1)(1-p_2)=1-p$. We then have that $G_p$ has the same distribution as $G[V_{p_1}\cup V_{p_2}]$, and $p_1\ge \frac{1+\epsilon-\epsilon^3}{d}$. By Theorem \ref{th: percolated path}, \textbf{whp} there is an induced path $G[V_{p_1}]$ of length at least $\Omega(\epsilon^2n/d)$. Similar to the above proof, one can consider $N_1,N_2$ the neighbourhoods in $G$ of some sufficiently long prefix and suffix of the path, which have \textbf{whp} at least $\epsilon^5dn$ edges between them. We can then consider sequentially every vertex in $N_1$, and whether it falls into $V_{p_2}$, noting that a typical vertex in $N_1$ will have $\Omega(d)$ neighbours in $N_2$. Upon reaching a sufficiently large subset $W\subseteq N_1$ in $V_{p_2}$ whose neighbourhood (in $G$) in $N_2$ is of order $d|W|$, we may percolate $N_2$ with probability $p_2$, and \textbf{whp} obtain an edge in $G[V_{p_2}]$ between $N_1$ and $N_2$, and then complete the proof as before.    
\end{remark}

\section{Non-isomorphic induced subgraphs}\label{sec: isomorphism}
The proof of Theorem \ref{th: isomorphism lower bound} will utilise Theorem \ref{th: percolated path}, together with Lemma \ref{expanding}.
\begin{proof}[Proof of Theorem \ref{th: isomorphism lower bound}]
Let $\epsilon,\delta(\epsilon)>0$ be as in the statement of Theorem \ref{th: percolated path}. We may assume that $\delta_1\le \delta(\epsilon)$, and in particular, $\delta_1\le \epsilon^{8/\epsilon^4}$. Let $p=\frac{1+\epsilon}{d}$. Then, by Theorem \ref{th: percolated path}, \textbf{whp} $G_p$ contains an induced path $P$ of length exactly $\frac{\epsilon^2n}{3d}$. Furthermore, by Lemma \ref{expanding} and by our assumption on $\delta_1$, \textbf{whp} $|N_G(P)|\ge (1-\epsilon^4)\left(\frac{\epsilon^2n}{3}-\frac{\epsilon^4n}{18}\right)\ge \frac{\epsilon^2n}{3}-\epsilon^4n$. Thus, there exists (deterministically) an induced path $P=\{v_1,v_2,\ldots,v_k\}$ in $G$ on $k=\frac{\epsilon^2n}{3d}$ vertices, such that $|N_G(P)|\ge \frac{\epsilon^2n}{3}-\epsilon^4n$.

Now, since $G$ is $d$-regular, we have that $\sum_{v\in N_G(P)}d(v,P)\le \sum_{v\in P}d(v)=\epsilon^2n/3$. Suppose towards contradiction that there are more than $2\epsilon^4n$ vertices in $N_G(P)$ each having at least two neighbours in $P$. Then, 
\begin{align*}
    \sum_{v\in N_G(P)}d(v,P)\ge |N_G(P)|+2\epsilon^4n>\epsilon^2n/3,
\end{align*}
which is a contradiction. Hence, there exists a set $U\subseteq N_G(P)$ of size at least $\frac{\epsilon^2n}{3}-3\epsilon^4n$, such that every $u\in U$ has exactly one neighbour in $P$. In particular, this implies that there are at least $\frac{\epsilon^2n}{10d}$ vertices in $P$ which have at least $\frac{d}{10}$ neighbours in $U$. Let us denote the set of these vertices in $P$ by $W$. Let us then choose a subset $W'\subseteq W$, such that $v_1,v_2,v_{k-1},v_k\notin W'$, and the distance in $P$ between any $u,u'\in W'$ is at least $4$. Note that we can choose such $W'$ with $|W'|\ge |W|/5\ge \frac{\epsilon^2n}{50d}$. Crucially, observe that every $u\in N_G(W')\cap U\coloneqq N_{W'}$ has a unique neighbour $v\in W'$ on $P$.

Now, let $H\coloneqq G[V(P)\cup N_{W'}]$. Let $A\neq A'$ be subsets of $N_{W'}$, such that $G[V(P)\cup A]\cong G[V(P)\cup A']$. Let $\psi:V(P)\cup A\to V(P)\cup A'$ be an isomorphism between these two graphs. We claim that $\psi\mid _{V(P)}$ is either the trivial automorphism of $P$, or the non-trivial involution automorphism of $P$. First, note that the only two vertices in $H$ which are of degree one and have a neighbour of degree two are $v_1$ and $v_k$. Indeed, $v_2,v_{k-1}\notin W'$ and thus have degree two, and any other vertex of degree one is in $N_{W'}$, and thus has a neighbour on $V(P)$ of degree at least three. It thus suffices to show that $\psi$ sends vertices of degree two in $V(P)$ to (possibly other) vertices of degree two in the same set $V(P)$.

Suppose towards contradiction that there exists $u\in V(P)$ such that $d_H(u)=2=d_{G[V(P)\cup A]}(u)$ and $\psi(u)\in A'$. Since every vertex in $A'\subseteq N_{W'}$ has a unique neighbour in $P$ and since $\psi(u)$ has degree two in $G[V(P)\cup A']$, we have that $\psi(u)$ has a neighbour in $A'$, which we denote by $v$. Let $x_1$ be the unique neighbour of $\psi(u)$ in $P$, and let $x_2$ be the unique neighbour of $v$ in $P$, noting that both $x_1$ and $x_2$ have degree at least three in $G[V(P)\cup A']$. We then have that $\psi(u)$ belongs to the path $x_1\psi(u)vx_2$ in $G[V(P)\cup A']\subseteq H$. However, by construction of $H$, the two closest vertices to $u$ which have degree at least $3$ in $G[V(P)\cup A]$ (and therefore, in $H$) are at distance at least four from each other --- contradiction. 

Therefore, any automorphism $\psi$ has that $\psi\mid_{V(P)}$ is one of the two involution automorphisms. Hence, for a fixed automorphism of $P$, any two subsets $A,A'\in N_{W'}$ having $|A'\cap N(v)\cap N_{W'}|\neq|A'\cap N(v)\cap N_{W'}|$ for some $v\in W'$ cannot span (together with $P$) isomorphic subgraphs. We thus conclude that, for every tuple $(0\leq s_v\leq d/10)_{v\in W'}$, there exists at most one \textit{other} tuple $(0\leq s'_v\leq d/10)_{v\in W'}$ satisfying the following: there exist two subsets $A,A'\subset N_{W'}$ such that $|A\cap N(v)\cap N_{W'}|=s_v$ and $|A'\cap N(v)\cap N_{W'}|=s'_v$ for every $v\in W'$ and $G[V(P)\cup A],G[V(P)\cup A']$ are isomorphic. Therefore,
$$
\mu(G)\geq\frac{1}{2}(d/10)^{|W'|}\geq \frac{1}{2}\left(\frac{d}{10}\right)^{\frac{\epsilon^2n}{50d}}\ge\exp\left(\frac{\delta_2n\log d}{d}\right),
$$
for small enough $\delta_2$, as required.
\end{proof}

\section{Constructions}\label{sec: construction}
Both the proof of Theorem \ref{th: construction} and of Theorem \ref{th: isomorphism upper bound} will follow from quite similar constructions. Let us first collect some notation and results about the spectra of graphs. We refer the reader to~\cite{spectra, zbMATH07864382} for a comprehensive study of the spectra of graphs and graph operations.

\begin{lemma}[see, e.g.,~\cite{spectra}]\label{spectra of complete}
    The eigenvalues of the complete graph on $n$ vertices are $n-1$ of multiplicity $1$, and $-1$ of multiplicity $n-1$. 
    The eigenvalue of the empty graph on $n$ vertices is zero with multiplicity $n$.
\end{lemma}

We will make use of the following graph operation. Let $r>0$ be an integer. Let $G$ and $H$ be two graphs, such that $H$ is $r$-regular. The \textit{lexicographic product} (denoted by $\mathrm{lex}(G,H)$) of $G$ and $H$ has vertex set $V(G)\times V(H)$ and for every $u,v\in V(G)$ and $x,y\in V(H)$, $(u,x)$ is adjacent to $(v,y)$ if and only if either $uv\in E(G)$ or $u=v\wedge xy\in E(H)$.

One can think of such an operation as taking a graph $G$ and replacing each vertex of $G$ by a copy of $H$; in particular, when $H$ is the empty graph, this is simply the blow-up operation. The lexicographic product satisfies the following property (see, e.g.,~\cite{spectra}).
\begin{lemma}\label{spectra of lex}
    Suppose $\lambda_1,\dots, \lambda_n$ are the eigenvalues of a graph $G$ on $n$ vertices and $\mu_1=r,\dots,\mu_m$ are the eigenvalues of an $r$-regular graph $H$ on $m$ vertices. Then, the eigenvalues of the adjacency matrix of the lexicographic product $\mathrm{lex}(G,H)$ are $\lambda_im+r$ of multiplicity $1$ for $1\leq i\leq n$ and $\mu_j$ of multiplicity $n$ for $2\leq j\leq m$.
\end{lemma}

Let us begin with the proof of Theorem \ref{th: isomorphism upper bound}, whose construction is slightly simpler.
\begin{proof}[Proof of Theorem \ref{th: isomorphism upper bound}]
Let $\delta>0$ be a constant. Let $d\in \mathbb{N}$ be sufficiently large, and let $n\in \mathbb{N}$ be such that $d,n$ satisfy the parity assumptions which are implicit below (in particular, $n$ is divisible by $d$). Let $d_0\coloneqq d_0(\delta)$ be the smallest integer satisfying $d_0\ge 3$ and $\sqrt{d_0}\ge \frac{4}{\delta}$. 

Assume first, for the sake of clarity of presentation, that $d$ is divisible by $d_0$. Let $n_0\coloneqq n\cdot d_0/d$, and let $H_1$ be an $(n_0, d_0, \lambda_0)$-graph with $\lambda_0\le 3\sqrt{d_0}$ --- indeed, such a graph exists, since a random $d_0$-regular graph on $n_0$ vertices typically satisfies this (see, e.g.,~\cite{KS06}). We note that we assume here that $n_0$ is sufficiently large with respect to $d_0$, and in turn, $n$ is sufficiently large with respect to $d$. Let $H_2$ be the empty graph on $d/d_0$ vertices. Set $G=\mathrm{lex}(H_1,H_2)$. By Lemmas \ref{spectra of complete} and \ref{spectra of lex} and by our choice of $d_0$, we have that $G$ is an $(n,d,\lambda)$-graph with
\begin{align*}
    \frac{\lambda}{d}=\frac{\lambda_0}{d_0}\le \frac{3}{\sqrt{d_0}}<\delta,
\end{align*}
and thus satisfies the assumption of Theorem \ref{th: isomorphism upper bound}.

Let $V(H_1)=\{v_1,\ldots, v_{n_0}\}$. Let $V_1,\ldots, V_{n_0}$ be subsets of $V(G)$, where $V_i$ is the blow-up of $v_i$. Note that, by construction, given two graphs $G_1,G_2\subseteq G$, if for every $i\in [n_0]$, $|V(G_1)\cap V_i|=|V(G_2)\cap V_i|$, then $G_1\cong G_2$. Therefore,
\begin{align*}
    \mu(G)\le \left(\frac{d}{d_0}+1\right)^{n_0}=\exp\left\{\frac{nd_0}{d}\log(d/d_0+1)\right\}\le \exp\left\{\frac{Cn\log d}{d}\right\},
\end{align*}
as required.

In the general case, when $d$ is not divisible by $d_0$, let us write $d=qd_0+r$ where $q\in \mathbb{N}$ and $1 \le r \le d_0-1$. Suppose first that $qr$ is even. We then set $H_2$ to be a graph on $q$ vertices, formed by taking a disjoint union of $\lfloor \frac{q}{r+1}\rfloor-1$ cliques of size $r+1$, and on the remaining $q-(r+1)\left(\lfloor \frac{q}{r+1}\rfloor-1\right)$ vertices we draw an arbitrary $r$-regular graph (indeed, such a graph exists since $qr$ and $r(r+1)$ are both even). If $qr$ is odd, then both $q$ and $r$ are odd, and thus $q-1$ is even. Let us write $d=(q-1)d_0+(d_0+r)$. We then set $H_2$ to be a graph on $q-1$ vertices, formed by taking a disjoint union $\lfloor\frac{q-1}{d_0+r+1}\rfloor-1$ cliques of size $d_0+r+1$, and on the remaining $(q-1)-(d_0+r+1)\left(\lfloor\frac{q-1}{d_0+r+1}\rfloor-1\right)$ vertices we draw an arbitrary $(d_0+r)$-regular graph (indeed, such a graph exists since $q-1$ and $(d_0+r)(d_0+r+1)$ are both even). The rest of the proof, for both cases (both choices of $H_2$), is quite similar to the case where $d$ is divisible by $d_0$.
\end{proof}

The construction for the proof of Theorem \ref{th: construction} is very similar, where instead of replacing every vertex with an independent set, we will replace every vertex with a copy of the complete graph.
\begin{proof}[Proof of Theorem \ref{th: construction}]
Let $\delta>0$ be a constant. Let $d\in \mathbb{N}$ be sufficiently large, and let $n\in \mathbb{N}$ be such that $d,n$ satisfy the parity assumptions which are implicit below (in particular, $n$ is divisible by $d$). Let $d_0$ be the smallest integer satisfying $d_0\ge 3$ and $\sqrt{d_0}\ge \frac{4}{\delta}$. 

Assume first, for the sake of clarity of presentation, that $d+1$ is divisible by $d_0+1$. Let $k\coloneqq \frac{d+1}{d_0+1}$, and let $n_0=n/k$. Let $H$ be an $(n_0,d_0,\lambda_0)$-graph satisfying that $\lambda_0\le 3\sqrt{d_0}$ (indeed, a random $d_0$-regular graph on $n_0$ vertices typically satisfies this). In particular, here too we assume that $n_0$ is sufficiently large with respect to $d_0$, and in turn, $n$ is sufficiently large with respect to $d$. Set $G=\mathrm{lex}(H,K_k)$. Noting that $kd_0+k-1=d$, we then have that $G$ is a $d$-regular graph on $kn_0=n$ vertices. Furthermore, by Lemmas \ref{spectra of complete} and \ref{spectra of lex} we have that the second largest eigenvalue of $G$, $\lambda$, satisfies that
\begin{align*}
    \frac{\lambda}{d}=\frac{k\lambda_0+k-1}{kd_0+k-1}\le \frac{4}{\sqrt{d_0}} \le \delta,
\end{align*}
and thus $G$ satisfies the assumptions of Theorem \ref{th: construction}.

By the construction of $G$, we have that a largest independent set of $G$ is of size at most $n_0=(d_0+1)n/d$. Therefore, a longest induced path in $G$ is of length at most $2(d_0+1)n/d$. Setting $C\coloneqq C(\delta)=2(d_0+1)$ completes the proof.

In the general case, where $d+1$ is not divisible by $d_0+1$, let $r$ be the residue of $d+1$ modulo $d_0+1$. Let $k\coloneqq \frac{d+1-r}{d_0+1}$, and let $H_2$ be a copy of $K_k$ with an $r$-regular graph removed. We then take $G=\mathrm{lex}(H,H_2)$. The rest of the proof is quite similar to the case where $d+1$ is divisible by $d_0+1$.
\end{proof}

\paragraph{Acknowledgements} Part of this work was done while the fourth author was visiting Tel Aviv University, and he would like to thank the university for its hospitality. The second author was supported in part by NSF-BSF grant 2023688.

\bibliographystyle{abbrv} 
\bibliography{Bib}

\end{document}